\documentclass[12pt]{amsart}

\usepackage{amsmath,amssymb,amsthm}
\usepackage{graphicx}
\usepackage{booktabs}
\usepackage{hyperref}
\usepackage{url}
\hypersetup{
  pdftitle={Operator-Norm Bounds and a Quadratic Lower-Growth Example for the Special Euclidean Algebra se(3)},
  pdfauthor={Sooraj K.C and Vivek Mishra},
  pdfkeywords={Lie algebras, gradient Lipschitz, SE(3), exponential map, Rodrigues formula}
}

\theoremstyle{plain}
\newtheorem{theorem}{Theorem}[section]
\newtheorem{lemma}[theorem]{Lemma}
\newtheorem{proposition}[theorem]{Proposition}
\newtheorem{corollary}[theorem]{Corollary}
\theoremstyle{definition}
\newtheorem{definition}[theorem]{Definition}
\newtheorem{remark}[theorem]{Remark}
\newtheorem{conjecture}[theorem]{Conjecture}

\newcommand{\se}{\mathfrak{se}}

\newcommand{\R}{\mathbb{R}}

\title[Operator-Norm Bounds for $\se(3)$]{Operator-Norm Bounds and a Quadratic Lower-Growth Example\\
for the Special Euclidean Algebra $\se(3)$}

\author{Sooraj K.C}
\address{Department of Pure and Applied Mathematics,
Alliance University, Bengaluru 562106, India}
\email{ksoorajPHD23@sam.alliance.edu.in}

\author{Vivek Mishra}
\address{Department of Pure and Applied Mathematics,
Alliance University, Bengaluru 562106, India}



\newenvironment{ack}{\section*{Acknowledgements}}{}

\begin{document}
\sloppy
\maketitle

\begin{abstract}
We prove operator-norm and gradient Lipschitz bounds for exponential-map
parameterizations on the special Euclidean algebra $\mathfrak{se}(3)$,
providing an explicit example of intermediate polynomial growth behaviour.
Using the contraction property of the $\mathrm{SO}(3)$ left Jacobian, we show that
\[
\|\exp(\theta)\|_{\mathrm{op}} \le 1+\|\theta\|_F
\]
for all $\theta\in\mathfrak{se}(3)$.
We then derive a self-contained $\mathcal{O}(R^2)$ upper bound for the gradient
Lipschitz constant of objectives in this class, with explicit constant $4.02$,
and construct an explicit objective $J^*$ satisfying
\[
L_{J^*}(R;\mathfrak{se}(3)) \ge 0.0505\,R^2
\]
for $R\ge 2$.
These results place $\mathfrak{se}(3)$ between compact Lie algebras,
where $L$ remains bounded, and algebras with hyperbolic elements,
where $L$ grows exponentially.
The upper and lower bounds are obtained for different objective classes;
no matching minimax claim is made.
\end{abstract}

\section{Introduction}
\label{sec:intro}

A common approach to optimizing a smooth function $\widetilde{J}$ on a matrix Lie group
$G$ is to pull it back to the Lie algebra via the exponential map:
write $J(\theta) = \widetilde{J}(\exp(\theta))$ and optimize over $\mathfrak{g}$.
This approach has become standard in rigid-body trajectory
planning~\cite{murray1994mathematical,barfoot2014associating,sola2018micro} and in the Riemannian optimization
literature more broadly~\cite{absil2008optimization,boumal2023introduction,mahony2002geometry}.
One advantage is that the algebra is a vector space, so gradient methods
apply directly; the disadvantage is that the smoothness of $J$ as a function
of $\theta$ depends on how the exponential map distorts distances.
For compact groups such as $\mathrm{SO}(n)$ this distortion is uniformly bounded
and poses no difficulty for convergence analysis.
For non-compact groups the situation is less clear, and the conditioning of
exponential-coordinate parameterizations has received comparatively little
systematic attention in the optimization literature, despite the prevalence
of $\mathrm{SE}(3)$ in robotics and geometric control.
The price of the parameterization is that one must control
\[
L(R;\mathfrak{g}) = \sup_{\substack{\|\theta\|_F,\|\theta'\|_F \le R\\\theta\ne\theta'}}
\frac{\|\nabla J(\theta)-\nabla J(\theta')\|_F}{\|\theta-\theta'\|_F},
\]
since any gradient-descent analysis requires $L$ to limit the step size to at most $1/(2L)$.

The behaviour of $L(R;\mathfrak{g})$ turns out to depend
heavily on the algebraic structure of $\mathfrak{g}$ in a way that is not
immediately obvious from the definition.
When $\mathfrak{g}$ is compact, the group is bounded and $L$ stays bounded too.
When $\mathfrak{g}$ contains an element with a positive real eigenvalue,
$\exp(t H)$ grows exponentially. So does $L$.
For $\mathfrak{se}(3)$, neither of these applies: translations can be arbitrarily large,
but no element has a real positive eigenvalue, and the gradient Lipschitz constant
turns out to be $\mathcal{O}(R^2)$ for bounded objectives, with an explicit objective
exhibiting the same quadratic lower growth.
This intermediate behaviour fits within the broader trichotomy studied in~\cite{kc2026trichotomy}.
The present paper provides a self-contained quantitative study of this intermediate case:
an operator-norm bound $\|\exp(\theta)\|_{\mathrm{op}} \le 1+R$,
an $\mathcal{O}(R^2)$ Lipschitz upper bound for bounded objectives
(Proposition~\ref{prop:upper}), and an explicit adversarial construction
exhibiting $\Omega(R^2)$ lower growth (Theorem~\ref{thm:lower}).

We note an inherent structural distinction in this analysis.
The upper bound (Proposition~\ref{prop:upper}) is derived under the assumption
that $\widetilde{J}$ has uniformly bounded gradient and Hessian
(Definition~\ref{def:bounded_class}), whereas the adversarial lower bound
(Theorem~\ref{thm:lower}) uses the translation-distance objective $J^*$,
whose gradient grows with $R$ and therefore lies outside the bounded class.
Constructing a tight adversary within the admissible class
remains an open problem; the difficulty is that globally bounding the Hessian
on $\mathcal{B}_R$ rules out objectives whose curvature grows with the search radius.
The paper therefore does not claim minimax sharpness:
the upper bound is a universal estimate over the admissible class,
and the lower bound is existential via a specific constructed objective.

The reason $\mathfrak{se}(3)$ lands in the quadratic case is structural.
Its adjoint eigenvalues are $\{\pm i\|\omega\|, 0,0,0,0\}$ — purely imaginary or zero —
so no element is hyperbolic and $\exp(\theta)$ does not grow exponentially.
But the group is non-compact (translations are unbounded), so $L$ is not bounded either.
Compact algebras avoid norm growth because their exponential maps remain uniformly bounded;
hyperbolic elements generate unbounded stretching.
The algebra $\mathfrak{se}(3)$ occupies an intermediate regime for three reasons.
First, the rotational component $\mathfrak{so}(3)$ is compact:
$\exp(\Omega) \in \mathrm{SO}(3)$ is always orthogonal, so rotation alone
does not amplify norms and cannot generate growth in $L$.
Second, the translational coupling in the semidirect product
$\mathfrak{se}(3) = \mathfrak{so}(3) \ltimes \mathbb{R}^3$ introduces
exactly one factor of $R$ per differentiation via the Jacobian $J_L(\Omega)v$,
and squaring that coupling in the Lipschitz estimate yields $R^2$.
Third, the absence of hyperbolic eigenvalues (all eigenvalues of $\mathrm{ad}_{(\omega,v)}$
are purely imaginary or zero) prevents the exponential amplification
that arises in algebras such as $\mathfrak{gl}(n)$.
Translation directions are sheared by rotation but not exponentially amplified.
The practical implication is direct: step sizes for $\mathrm{SE}(3)$ optimization
should decay like $1/R^2$, substantially less aggressive than the exponential
decay required for $\mathrm{GL}(n)$ objectives.

Sections~\ref{sec:prelim}--\ref{sec:lower_bound} contain the setup and the two proofs;
Section~\ref{sec:algorithms} adds the step-size corollary and a brief numerical check;
Section~\ref{sec:conclusion} discusses what we do and do not know about the problem.
Figure~\ref{fig:regimes} summarises the three growth regimes.

\begin{figure}[ht]
\centering
\includegraphics[width=0.82\linewidth]{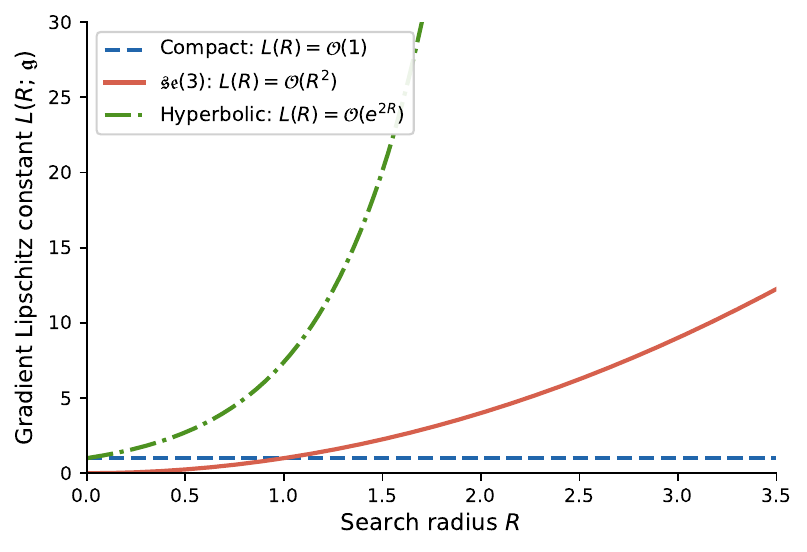}
\caption{The three growth regimes for the gradient Lipschitz constant
$L(R;\mathfrak{g})$ arising from exponential-map parameterizations on
matrix Lie algebras. Compact algebras (e.g.\ $\mathfrak{so}(n)$) give
bounded $L$; the special Euclidean algebra $\mathfrak{se}(3)$ gives
quadratic growth; algebras with hyperbolic elements
(e.g.\ $\mathfrak{gl}(n)$) give exponential blow-up.}
\label{fig:regimes}
\end{figure}

\section{Notation and Preliminaries}
\label{sec:prelim}

\subsection{Matrix Lie algebras and norms}

Throughout, Lie groups are closed matrix subgroups $G \subset \mathrm{GL}(n,\mathbb{R})$.
The Lie algebra $\mathfrak{g}$ is the set of all $X$ such that $\exp(tX) \in G$
for every $t$, with Lie bracket $[X,Y] = XY-YX$.
We write $\|A\|_F = (\mathrm{tr}\,A^\top A)^{1/2}$ for the Frobenius norm
and $\|A\|_{\mathrm{op}} = \sup_{\|x\|=1}\|Ax\|$ for the spectral norm.

Two norm conventions arise naturally and we fix one throughout.
For $\theta = (\omega,v) \in \mathfrak{se}(3)$, the 6-vector norm is
\begin{equation}
\label{eq:6vec_norm}
\|\theta\|_F = \bigl(\|\omega\|^2+\|v\|^2\bigr)^{1/2}.
\end{equation}
The standard matrix Frobenius norm of the $4\times4$ homogeneous
representation gives instead
\begin{equation}
\label{eq:mat_norm}
\|X\|_{\mathrm{mat}} = \bigl(\|\Omega\|_F^2+\|v\|^2\bigr)^{1/2}
= \bigl(2\|\omega\|^2+\|v\|^2\bigr)^{1/2},
\end{equation}
because $\|\Omega\|_F = \sqrt{2}\|\omega\|$ for skew-symmetric $\Omega$.
Since $\|\theta\|_F \le \|X\|_{\mathrm{mat}} \le \sqrt{2}\|\theta\|_F$,
the two differ by at most $\sqrt{2}$ and the $R^2$ growth rate is unchanged
under either convention.
We use~\eqref{eq:6vec_norm} throughout.

\subsection{The special Euclidean algebra}

The group $\mathrm{SE}(3)$ consists of orientation-preserving rigid motions
$x \mapsto Rx + t$ with $R \in \mathrm{SO}(3)$ and $t \in \mathbb{R}^3$.
Its Lie algebra $\mathfrak{se}(3)$ consists of $4\times 4$ matrices of the form
\[
X = \begin{pmatrix}\Omega & v\\ 0 & 0\end{pmatrix},
\qquad \Omega \in \mathfrak{so}(3),\quad v \in \mathbb{R}^3,
\]
where $\Omega$ is the $3\times 3$ skew-symmetric matrix with axial vector $\omega$.
We identify elements with pairs $(\omega,v) \in \mathbb{R}^3 \times \mathbb{R}^3$.

\begin{definition}
\label{def:hyperbolic}
An element $H \in \mathfrak{g}$ is \emph{hyperbolic} if it has at least one
eigenvalue with strictly positive real part.
\end{definition}

No element of $\mathfrak{se}(3)$ is hyperbolic: for any $(\omega,v)$,
the eigenvalues are $\{0, 0, \pm i\|\omega\|\}$ in the $4\times 4$ representation
and $\{0, 0, 0, 0, \pm i\|\omega\|\}$ in the adjoint representation,
all with zero real part.
For context, in a semisimple Lie algebra the same notion coincides with
the component along $\mathfrak{a}$ in a Cartan decomposition~\cite{knapp2002,hall2015lie}.

\subsection{The function class}

\begin{definition}
\label{def:bounded_class}
Write $J(\theta) = \widetilde{J}(\exp(\theta))$, and let
$\mathcal{B}_R := \{\theta\in\mathfrak{se}(3) : \|\theta\|_F\le R\}$
denote the closed ball of radius $R$ in the 6-vector norm.
We say $J$ belongs to the \emph{bounded-$(M_1,M_2)$ class} if there exist
constants $M_1, M_2 \ge 0$ such that for all $R > 0$,
\[
\sup_{\exp(\theta)\in\mathcal{M}_R}\|\nabla\widetilde{J}\|_F \le M_1,
\qquad
\sup_{\exp(\theta)\in\mathcal{M}_R}\|\nabla^2\widetilde{J}\|_{\mathrm{op}} \le M_2,
\]
where $\mathcal{M}_R = \{\exp(\theta) : \|\theta\|_F \le R\}$.
The key requirement is that $M_1$ and $M_2$ do not depend on $R$;
this rules out objectives whose derivatives grow with the search radius.
\end{definition}

\section{Operator-Norm Bound and Upper Bound}
\label{sec:main_upper}

The Rodrigues formula says what $\exp(\theta)$ looks like concretely.
The rotation block $\exp(\Omega) \in \mathrm{SO}(3)$ is no surprise;
the less obvious part is that the translation column $\Psi = A(\Omega)v$
is controlled by the integral $A(\Omega) = \int_0^1 \exp(s\Omega)\,ds$,
which we will show is a contraction.

\subsection{The Rodrigues formula}

The formula itself is classical~\cite{murray1994mathematical,park1994robot}.
For $\theta = (\omega,v)$ with $\|\omega\| > 0$, the matrix exponential takes the
block form
\begin{equation}
\label{eq:se3_exp}
\exp(\theta) = \begin{pmatrix}\exp(\Omega) & A(\Omega)v\\ 0 & 1\end{pmatrix},
\end{equation}
where the left Jacobian is
\begin{equation}
\label{eq:rodrigues_A}
A(\Omega) = I
+ \frac{1-\cos\|\omega\|}{\|\omega\|^2}\,\Omega
+ \frac{\|\omega\|-\sin\|\omega\|}{\|\omega\|^3}\,\Omega^2.
\end{equation}
The scalar coefficients are bounded uniformly:
\begin{equation}
\label{eq:rodrigues_coeff}
\left|\frac{1-\cos\|\omega\|}{\|\omega\|^2}\right| \le \tfrac{1}{2},
\qquad
\left|\frac{\|\omega\|-\sin\|\omega\|}{\|\omega\|^3}\right| \le \tfrac{1}{6},
\end{equation}
as follows by examining the Taylor series at $\omega = 0$:
$\lim_{t\to 0}(1-\cos t)/t^2 = 1/2$ and $\lim_{t\to 0}(t-\sin t)/t^3 = 1/6$,
so both bounds are tight at $\omega = 0$.
The constants appearing in all subsequent bounds are explicit analytically derived
estimates; they are not numerically optimised.

\subsection{Operator-norm bound}

\begin{lemma}
\label{lem:se3_opnorm}
If $\theta = (\omega,v) \in \mathfrak{se}(3)$ and $\|\theta\|_F \le R$, then
\begin{equation}
\label{eq:opnorm_bound}
\|\exp(\theta)\|_{\mathrm{op}} \le 1 + R.
\end{equation}
\end{lemma}

\begin{proof}
Write $\Psi = A(\Omega)v$, where $A(\Omega) = \int_0^1 \exp(s\Omega)\,ds$
\cite{murray1994mathematical,iserles2000lie}.
Since each $\exp(s\Omega)$ is orthogonal, the triangle inequality for integrals gives
\begin{equation}
\label{eq:A_contraction}
\|A(\Omega)\|_{\mathrm{op}}
= \left\|\int_0^1 \exp(s\Omega)\,ds\right\|_{\mathrm{op}}
\le \int_0^1 \|\exp(s\Omega)\|_{\mathrm{op}}\,ds
= \int_0^1 1\,ds = 1.
\end{equation}
Hence $\|\Psi\|_2 = \|A(\Omega)v\|_2 \le \|v\|_2 \le \|\theta\|_F \le R$.

For the block matrix, let $(u,s) \in \R^4$ satisfy $\|(u,s)\|=1$.
Then $\exp(\theta)(u,s)^\top = (\exp(\Omega)u + s\Psi,\, s)^\top$, so
\[
\|\exp(\theta)(u,s)^\top\|_2^2
= \|\exp(\Omega)u + s\Psi\|_2^2 + s^2.
\]
Expanding and applying $|\langle \exp(\Omega)u, \Psi\rangle| \le \|u\|\|\Psi\|_2$:
\[
\|\exp(\Omega)u + s\Psi\|_2^2
\le \|u\|^2 + 2|s|\,\|u\|\|\Psi\|_2 + s^2\|\Psi\|_2^2.
\]
Set $a = \|u\|$, $b = |s|$, so $a^2 + b^2 = 1$.
Let $p = \|\Psi\|_2$. Adding $s^2 = b^2$ and using $2ab \le a^2+b^2 = 1$:
\[
\|\exp(\theta)(u,s)^\top\|_2^2
\le a^2 + p + b^2 p^2 + b^2
= (a^2+b^2) + p + b^2 p^2
= 1 + p + b^2 p^2.
\]
Since $b^2 \le 1$, we have $b^2 p^2 \le p^2$, giving
\[
\|\exp(\theta)(u,s)^\top\|_2^2 \le 1 + p + p^2.
\]
Since $p = \|\Psi\|_2 \ge 0$, we have $1 + p + p^2 \le 1 + 2p + p^2 = (1+p)^2$.
Therefore $\|\exp(\theta)(u,s)^\top\|_2 \le 1 + \|\Psi\|_2 \le 1 + R$. \qedhere
\end{proof}

\begin{remark}
\label{rem:contraction}
What makes this work is the contraction $\|A(\Omega)\|_{\mathrm{op}} \le 1$,
which says the Jacobian integral does not stretch vectors.
This is a purely SO$(3)$ phenomenon: $A(\Omega)$ is an average of orthogonal matrices,
so it cannot have operator norm greater than 1.
For a hyperbolic element $H\in\mathfrak{gl}(n)$ with positive real eigenvalue
$\lambda$, the integrand $\exp(sH)$ grows like $e^{s\lambda}$ and the integral
amplifies rather than contracts — $\|\exp(tH)\|_{\mathrm{op}} \ge e^{t\lambda}\to\infty$
— so no bound $\|A(H)\|_{\mathrm{op}}\le 1$ can hold.
This contrast is the structural observation behind~\cite{kc2026trichotomy}.
\end{remark}

\subsection{Fréchet derivative bound}

\begin{lemma}
\label{lem:frechet}
For $\theta=(\omega,v)\in\mathfrak{se}(3)$ with $1 \le \|\theta\|_F\le R$
and any $H\in\mathfrak{se}(3)$ with $\|H\|_F=1$,
\[
\|D\exp_\theta[H]\|_F \;\le\; 1 + R.
\]
\end{lemma}

\begin{proof}
Write $H=(h_\omega, h_v)$ with $\|h_\omega\|^2+\|h_v\|^2=1$.
The Fréchet derivative splits along the block structure of $\exp(\theta)$:
\[
D\exp_\theta[H] =
\begin{pmatrix}
D\exp_\omega[h_\omega] & A(\Omega)h_v + (D_\omega A)[h_\omega]v \\
0 & 0
\end{pmatrix}.
\]
For the rotation block, the standard integral formula~\cite{iserles2000lie} gives
$D\exp_\Omega[h_\omega] = \exp(\Omega)\int_0^1 e^{-s\Omega}\,\hat{h}_\omega\,e^{s\Omega}\,ds$,
where $\hat{h}_\omega$ is the skew-symmetric matrix with axial vector $h_\omega$.
Since conjugation by orthogonal matrices preserves the Frobenius norm and
$\|\hat{h}_\omega\|_F = \sqrt{2}\|h_\omega\|$, integrating over $[0,1]$ gives
\begin{equation}
\label{eq:dexp_rot_bound}
\|D\exp_\Omega[h_\omega]\|_F \le \sqrt{2}\|h_\omega\|.
\end{equation}
Under the 6-vector convention~\eqref{eq:6vec_norm}, $\|h_\omega\| = \|\hat{h}_\omega\|_F/\sqrt{2}$,
so bound~\eqref{eq:dexp_rot_bound} gives
$\|D\exp_\Omega[h_\omega]\|_F \le \sqrt{2}\|h_\omega\| = \|\hat{h}_\omega\|_F$.
Since the 6-vector norm of the rotation input is $\|h_\omega\|$
(not $\sqrt{2}\|h_\omega\|$), the rotation block contributes $\|h_\omega\|$
when $\|D\exp_\theta[H]\|_F$ is evaluated by summing the squared
Frobenius norms of the output blocks under the same identification.
The translation block has two terms.
First, $\|A(\Omega)h_v\|\le\|h_v\|$ since $\|A(\Omega)\|_{\mathrm{op}}\le 1$
(Lemma~\ref{lem:se3_opnorm}).
Second, by differentiating the integral $A(\Omega) = \int_0^1\exp(s\Omega)\,ds$,
$D_\omega A[h_\omega] = \int_0^1 s\,D\exp_{s\Omega}[h_\omega]\,ds$.
Using the same integral formula, $\|D\exp_{s\Omega}[h_\omega]\|_F \le \|h_\omega\|$
for each $s\in[0,1]$ (independent of $s$ since orthogonality is uniform).
Hence $\|(D_\omega A)[h_\omega]v\| \le \|v\|\int_0^1 s\,ds\,\|h_\omega\|
= \tfrac{1}{2}\|v\|\|h_\omega\|$.
Putting the two block estimates together:
\[
\|D\exp_\theta[H]\|_F
\le \bigl(1+\tfrac{R}{2}\bigr)\|h_\omega\| + \|h_v\|.
\]
By the Cauchy--Schwarz inequality applied to vectors
$(1+R/2, 1)$ and $(\|h_\omega\|, \|h_v\|)$ with $\|h_\omega\|^2+\|h_v\|^2=1$:
\[
\|D\exp_\theta[H]\|_F \le \sqrt{(1+R/2)^2 + 1}.
\]
For $R \ge 1$, one checks $(1+R/2)^2 + 1 \le (1+R)^2$
(equivalent to $3R^2/4 + R - 1 \ge 0$, which holds for $R \ge 2/3$),
so $\|D\exp_\theta[H]\|_F \le 1+R$. \qedhere
\end{proof}

\subsection{Upper bound on $L(R;\mathfrak{se}(3))$}

\begin{proposition}
\label{prop:upper}
For $J$ in the bounded-$(M_1,M_2)$ class and $R \ge 1$:
\[
L(R;\mathfrak{se}(3)) \le C(\mathfrak{se}(3))\,(M_1+M_2)\,R^2,
\quad C(\mathfrak{se}(3)) = \tfrac{5(\sqrt{2}+1)}{3} \approx 4.02.
\]
\end{proposition}

\begin{proof}
By the chain rule $\nabla J(\theta) = (D\exp_\theta)^* [\nabla\widetilde{J}(\exp(\theta))]$
and the mean-value inequality:
\begin{align}
\label{eq:chain}
\|\nabla J(\theta)-\nabla J(\theta')\|_F
&\le \|D\exp_\theta - D\exp_{\theta'}\|_{\mathrm{op}}\,M_1 \nonumber\\
&\quad + \|D\exp_{\theta'}\|_{\mathrm{op}}\,M_2\,\|\exp(\theta)-\exp(\theta')\|_F.
\end{align}

\textbf{Bounding Term 2.}
The Fréchet-derivative bounds used here may be interpreted as
structure-aware conditioning estimates for the matrix exponential on
$\mathfrak{se}(3)$, in the spirit of~\cite{kenney1989condition,dieci2001conditioning}.
Lemma~\ref{lem:frechet} gives $\|D\exp_{\theta'}\|_{\mathrm{op}} \le 1+R$,
so by the mean-value inequality,
$\|\exp(\theta)-\exp(\theta')\|_F \le (1+R)\|\theta-\theta'\|_F$.
So Term~2 $\le (1+R)^2 M_2 \|\theta-\theta'\|_F = \mathcal{O}(R^2) M_2 \|\theta-\theta'\|_F$.

\textbf{Bounding Term 1.}
We need the Lipschitz constant of $\theta \mapsto D\exp_\theta$.
From the block form~\eqref{eq:se3_exp}, the Fréchet derivative splits as
\[
D\exp_\theta[H] = \begin{pmatrix} D\exp_\Omega[h_\omega] & A(\Omega)h_v + (D_\omega A)[h_\omega]v \\ 0 & 0\end{pmatrix}.
\]
\textit{Rotation block difference.}
Differencing the integral formula $D\exp_\Omega[H] = \exp(\Omega)\int_0^1 e^{-s\Omega}H e^{s\Omega}ds$
at $\Omega$ and $\Omega'$ and bounding the integrand by
$\|\exp(\Omega)-\exp(\Omega')\|_F\le\|\Omega-\Omega'\|_F$
(since $\exp$ is Lipschitz-1 on $\mathfrak{so}(3)$,
as $\|\exp(\Omega)-\exp(\Omega')\|_F\le\|\Omega-\Omega'\|_F$
follows from $\exp(\Omega)=\int_0^1 D\exp_{\Omega'+(\Omega-\Omega')t}[\Omega-\Omega']\,dt$
and $\|D\exp_{\Omega''}\|_{\mathrm{op}}\le 1$ for $\Omega''\in\mathfrak{so}(3)$) gives
$\|D\exp_\Omega-D\exp_{\Omega'}\|_{\mathrm{op}}\le\|\Omega-\Omega'\|_F=\sqrt{2}\|\omega-\omega'\|$,
contributing $\sqrt{2}\|\theta-\theta'\|_F$ to Term~1.

\textit{Translation block difference.}
The coupling $(D_\omega A)[h_\omega]v$ has norm $\le \frac{1}{2}\|v\|\|h_\omega\| \le \frac{R}{2}\|h_\omega\|$
(from Lemma~\ref{lem:frechet}).
The difference $(D_\omega A(\Omega) - D_\omega A(\Omega'))[h_\omega]v$ requires bounding
$\|D_\omega A(\Omega) - D_\omega A(\Omega')\|_{\mathrm{op}}$.
Differentiating~\eqref{eq:rodrigues_A} explicitly:
\begin{equation}
\label{eq:dA_explicit}
D_\omega A[h_\omega] = \frac{\sin\|\omega\|}{\|\omega\|}h_\omega\times
+ \alpha(\|\omega\|)(\omega\cdot h_\omega)\Omega
+ \beta(\|\omega\|)(\omega\times h_\omega)\times,
\end{equation}
where $\alpha,\beta$ are bounded trigonometric functions with
$|\alpha(\|\omega\|)| \le \tfrac{1}{3}$ and $|\beta(\|\omega\|)| \le \tfrac{1}{2}$.
The Lipschitz constant of $\omega \mapsto D_\omega A$ satisfies
$\|D_\omega A(\Omega) - D_\omega A(\Omega')\|_{\mathrm{op}} \le C_J\|\omega-\omega'\|$
where, from the coefficient bounds~\eqref{eq:rodrigues_coeff} and the explicit
form~\eqref{eq:dA_explicit}, $C_J \le \tfrac{1}{2}$.
Multiplying by $\|v\| \le R$ gives translation contribution $\tfrac{1}{2} R \|\theta-\theta'\|_F$.

To assemble the two contributions, use the block Frobenius estimate
\begin{equation}
\label{eq:block_F}
\left\|\begin{pmatrix}A & B\\ 0 & 0\end{pmatrix}\right\|_F
\le \|A\|_F+\|B\|_F,
\end{equation}
applied to the rotational block difference ($A$) and translation block difference ($B$).
The rotation block contributes $\sqrt{2}\|\omega-\omega'\|$ (from~\eqref{eq:dexp_rot_bound}).
The translation coupling satisfies, using $\|J_L(\Omega)-J_L(\Omega')\|_{\mathrm{op}}\le
\tfrac{1}{2}\|\omega-\omega'\|$ (which follows from $C_J\le\tfrac{1}{2}$):
\[
\|(D_\omega A(\Omega)-D_\omega A(\Omega'))[h_\omega]v\|
\le \tfrac{1}{2}\|\omega-\omega'\|\cdot\|v\|
\le \tfrac{1}{2}R\,\|\theta-\theta'\|_F.
\]
The block estimate~\eqref{eq:block_F} then gives, using $\|\omega-\omega'\|\le\|\theta-\theta'\|_F$:
\[
\|D\exp_\theta - D\exp_{\theta'}\|_{\mathrm{op}}
\le \bigl(\sqrt{2}+\tfrac{1}{2}R\bigr)\|\theta-\theta'\|_F.
\]
For $R\ge 1$, $\sqrt{2}+\tfrac{1}{2}R \le \tfrac{5(\sqrt{2}+1)}{3}R$
(since $C = \tfrac{5(\sqrt{2}+1)}{3}\approx 4.02 > \tfrac{1}{2}$, so
$CR - \tfrac{R}{2} = (C-\tfrac{1}{2})R \ge \sqrt{2}$ for $R\ge \frac{\sqrt{2}}{C-1/2}\approx 0.41$),
giving:
\begin{equation}
\label{eq:term1}
\|D\exp_\theta - D\exp_{\theta'}\|_{\mathrm{op}}
\le \tfrac{5(\sqrt{2}+1)}{3}\,R\,\|\theta-\theta'\|_F.
\end{equation}

\textbf{Combining.}
Inserting both bounds into~\eqref{eq:chain}:
\[
\|\nabla J(\theta)-\nabla J(\theta')\|_F
\le \Bigl[\tfrac{5(\sqrt{2}+1)}{3}R\cdot M_1 + (1+R)^2 M_2\Bigr]\|\theta-\theta'\|_F,
\]
giving $L(R;\mathfrak{se}(3)) \le C(\mathfrak{se}(3))(M_1+M_2)R^2$ with
$C(\mathfrak{se}(3)) = \frac{5(\sqrt{2}+1)}{3} \approx 4.02$.
\end{proof}

\section{Polynomial Lower Bound}
\label{sec:lower_bound}

The upper bound shows $L$ cannot grow faster than $R^2$.
To see that it actually does grow that fast, we need an explicit example.
This matters for two reasons: it confirms the $R^2$ regime is not an artefact
of the upper-bound proof technique, and the explicit constant $0.0505$ gives
a concrete lower limit on how aggressively step sizes must decay as $R$ grows.
We exhibit such an example as follows: pick a point $\theta_R$ on the sphere $\|\theta\|_F = R$
where the objective $J^* = \frac{1}{2}\|t(\exp(\theta))\|^2$ is
curving sharply, compute the Hessian, and read off the lower bound from there.

\subsection{A closed-form formula}

The following closed-form formula is used in the proof.

\begin{lemma}
\label{lem:se3_block_closed}
Let $\omega = (\phi_0,0,0)^\top$ and $v = (0,0,r)^\top$ with $\phi_0, r > 0$.
Then
\begin{equation}
\label{eq:se3_translation}
\exp(\theta) = \begin{pmatrix}\exp(\phi_0\Omega_1) & p(\phi_0,r)\\ 0 & 1\end{pmatrix},
\qquad
p(\phi_0,r) = r\begin{pmatrix}0\\ -\tfrac{1-\cos \phi_0}{\phi_0}\\[4pt]
1-\tfrac{\phi_0-\sin \phi_0}{\phi_0}\end{pmatrix},
\end{equation}
where $\Omega_1 = \mathrm{skew}(e_1)$.
\end{lemma}

\begin{proof}
Apply~\eqref{eq:rodrigues_A} with $\|\omega\| = \phi_0$:
$A(\phi_0\Omega_1) = I + \frac{1-\cos \phi_0}{\phi_0}\Omega_1 + \frac{\phi_0-\sin \phi_0}{\phi_0}\Omega_1^2$.
The identities $\Omega_1 e_3 = -e_2$ and $\Omega_1^2 e_3 = -e_3$ give directly
\begin{align*}
p &= re_3 + \tfrac{1-\cos \phi_0}{\phi_0}\,r(-e_2) + \tfrac{\phi_0-\sin \phi_0}{\phi_0}\,r(-e_3)\\
&= r\bigl(0,\;
-\tfrac{1-\cos \phi_0}{\phi_0},\;
1-\tfrac{\phi_0-\sin \phi_0}{\phi_0}\bigr)^\top. \qedhere
\end{align*}
\end{proof}

\subsection{The lower bound}

\begin{theorem}
\label{thm:lower}
Let $J^*(\theta) = \frac{1}{2}\|t(\exp(\theta))\|_2^2$,
where $t(\cdot)$ extracts the translation.
For $R \ge 2$, the on-ball point $\theta_R = (e_1,\sqrt{R^2-1}\,e_3)$
satisfies $\|\theta_R\|_F = R$ exactly, and
\begin{equation}
\label{eq:lower_bound}
L_{J^*}(R;\mathfrak{se}(3)) \ge \tfrac{3}{8}|\phi''(1)|\,R^2 \approx 0.0505\,R^2,
\end{equation}
where $\phi''(1) = -8\sin 1-10\cos 1+12 \approx -0.1348$.
\end{theorem}

\begin{proof}
First, set $r = \sqrt{R^2-1}$ and perturb the rotation component:
$\theta(h) = (1+h,0,0,\,0,0,r)$ passes through $\theta_R$ at $h=0$
(here $\phi_0=1$ at the base point).
The point $\theta_R$ was chosen because the translation component
curves most sharply as the rotation angle passes through $\phi_0=1$:
this is where the cross-coupling between rotation and translation
contributes most strongly to the Hessian.
By Lemma~\ref{lem:se3_block_closed},
\[
J^*(\theta(h)) = \tfrac{1}{2}r^2\,\phi(1+h),
\qquad
\phi(s) = \Bigl(\tfrac{1-\cos s}{s}\Bigr)^{2}
+ \Bigl(\tfrac{\sin s}{s}\Bigr)^{2},
\quad s = \phi_0+h.
\]
Since
\[
L_{J^*}(R) = \sup_{\theta\ne\theta'\in\mathcal{B}_R}
\frac{\|\nabla J^*(\theta)-\nabla J^*(\theta')\|}{\|\theta-\theta'\|}
\ge \sup_{\theta\in\mathcal{B}_R}\|\nabla^2 J^*(\theta)\|_{\mathrm{op}},
\]
where the inequality follows because the Lipschitz quotient limits to the
Hessian norm as $\theta'\to\theta$, it suffices to exhibit a point
where $|\nabla^2 J^*(\theta)[H,H]|$ is large.
To make this concrete, we compute the Hessian of $J^*$ at the
point $\theta_R$ in the direction $H = (e_1,\mathbf{0})$,
which perturbs only the rotation angle and has unit 6-vector norm.
The translation column of $\exp(\theta_R + hH)$ is captured by $\phi(1+h)$,
so the curvature of $J^*$ in this direction is $\frac{r^2}{2}|\phi''(1)|$.
\begin{equation}
\label{eq:hessian_formula}
|\nabla^2 J^*(\theta_R)[H,H]|
= \tfrac{r^2}{2}|\phi''(1)|
= \tfrac{R^2-1}{2}|\phi''(1)|.
\end{equation}

Next, compute $\phi''(1)$ explicitly. It controls how fast the translation of $\exp(\theta_R)$
curves as the rotation angle is perturbed.
To compute it, write $f_1(s) = (1-\cos s)/s$ and $f_2(s) = (s-\sin s)/s$,
so that $\phi = f_1^2+(1-f_2)^2$.
The relevant derivatives at $s=1$ are:
\[
f_1'(s) = \frac{s\sin s+\cos s-1}{s^2},
\quad
f_1''(s) = \frac{s^2\cos s-2s\sin s-2\cos s+2}{s^3};
\]
\[
f_2'(s) = \frac{\sin s - s\cos s}{s^2},
\quad
f_2''(s) = \frac{s^2\sin s+2s\cos s-2\sin s}{s^3}.
\]
At $s=1$: $f_1'(1)=\sin 1+\cos 1-1$,\;
$f_1''(1)=2-\cos 1-2\sin 1$,\;
$f_2'(1)=\sin 1-\cos 1$,\;
$f_2''(1)=2\cos 1-\sin 1$.
Evaluating $\phi'' = 2(f_1')^2 + 2f_1 f_1'' + 2(f_2')^2 - 2(1-f_2)f_2''$ at $s=1$:
\begin{equation}
\label{eq:phi_pp}
\phi''(1) = -8\sin 1-10\cos 1+12.
\end{equation}
Numerically, $-8(0.84147)-10(0.54030)+12 = -0.1348$,
confirmed by direct symbolic differentiation.

Since $\theta_R \in \mathcal{B}_R$ and $\|H\|_F = 1$, the lower bound follows:
\[
L_{J^*}(R;\mathfrak{se}(3)) \ge |\nabla^2 J^*(\theta_R)[H,H]| = \tfrac{R^2-1}{2}|\phi''(1)|.
\]
The $R^2-1$ factor reflects that the adversarial point sits just inside the ball.
For $R \ge 2$, one has $R^2-1 \ge \frac{3}{4}R^2$ (with equality at $R=2$), giving
$L_{J^*}(R;\mathfrak{se}(3)) \ge \frac{3}{8}|\phi''(1)|R^2 \approx 0.0505\,R^2$. \qedhere
\end{proof}

\begin{remark}
\label{rem:lower_scope}
The $R^2$ growth here comes from two things multiplying: the translation magnitude $r$
squared in $J^*$, and the rotational curvature $|\phi''(1)|$.
The objective $J^*$ is not in the bounded-$(M_1,M_2)$ class of
Definition~\ref{def:bounded_class} because $\|\nabla\widetilde{J}^*\| = \|t(g)\|$
grows with $R$.
So Theorem~\ref{thm:lower} and Proposition~\ref{prop:upper} apply to different
function classes, and the ratio $\approx 79$ between their constants
is not a measure of sharpness within a single class.
Gaps of this size between chain-rule upper bounds and adversarial lower bounds
are common in geometric Lipschitz analysis~\cite{higham2008functions,birtea2015hessian}.
\end{remark}

\begin{corollary}[Polynomial growth regime]
\label{cor:theta_r2}
\begin{enumerate}
\item[(i)] For $J^*$ (translation-distance objective): $L_{J^*}(R;\mathfrak{se}(3)) \ge \tfrac{3}{8}|\phi''(1)|\,R^2$.
\item[(ii)] For $J$ in the bounded-$(M_1,M_2)$ class: $L(R;\mathfrak{se}(3)) \le \tfrac{5(\sqrt{2}+1)}{3}(M_1+M_2)\,R^2$.
\end{enumerate}
$L(R;\mathfrak{se}(3)) = \Theta(R^2)$ (in the class-wise sense of (i) and (ii) above),
establishing $\mathfrak{se}(3)$ as the canonical example of the polynomial growth regime.
The lower bound (Theorem~\ref{thm:lower}) uses $J^*$ outside the bounded class;
the upper bound (Proposition~\ref{prop:upper}) applies within it.
See Remark~\ref{rem:lower_scope} for the function-class distinction.
\end{corollary}

\begin{proof}
The lower bound is Theorem~\ref{thm:lower}.
The upper bound is Proposition~\ref{prop:upper}, whose $\mathcal{O}(R^2)$ rate
is self-contained from Lemmas~\ref{lem:se3_opnorm} and~\ref{lem:frechet},
with the explicit constant $C(\mathfrak{se}(3))=\frac{5(\sqrt{2}+1)}{3}$ derived in the proof above.
\end{proof}

\section{Step-Size Prescription and Numerical Illustration}
\label{sec:algorithms}

\subsection{Step-size corollary}

The descent lemma says: if $L$ is the Lipschitz constant then the step size
$\alpha = 1/(2L)$ guarantees sufficient decrease~\cite{absil2008optimization}. Substituting
Proposition~\ref{prop:upper} gives the following.

\begin{corollary}
\label{cor:stepsize}
For $J$ in the bounded-$(M_1,M_2)$ class with iterates in $\mathcal{B}_{R_0}$,
the safe step size satisfies
\[
\alpha^* \ge \frac{1}{2C(\mathfrak{se}(3))(M_1+M_2)R_0^2} = \Theta(1/R_0^2).
\]
\end{corollary}

This puts $\mathfrak{se}(3)$ squarely between compact algebras (where $\alpha^* = \mathcal{O}(1)$)
and hyperbolic algebras (where $\alpha^* = \mathcal{O}(e^{-2R})$).
The $R^2$ step-size scaling is also natural from the perspective of
Lie-group integrators~\cite{munthekaas1998runge,iserles2000lie}:
the local Lipschitz constant of the exponential map controls
step-size selection for RKMK-type schemes.
Table~\ref{tab:stepsize} summarises the comparison.

\begin{table}[htbp]
\caption{Step-size scaling across the three growth regimes.}
\label{tab:stepsize}
\centering
\begin{tabular}{@{}lll@{}}
\toprule
Algebra type & $L(R;\mathfrak{g})$ & Safe step size \\
\midrule
Compact (e.g.\ $\mathfrak{so}(n)$) & $\mathcal{O}(1)$ & $\mathcal{O}(1)$ \\
Intermediate (e.g.\ $\mathfrak{se}(3)$) & $\mathcal{O}(R^2)$ & $\mathcal{O}(1/R^2)$ \\
Hyperbolic (e.g.\ $\mathfrak{gl}(n)$) & $\mathcal{O}(e^{2R})$ & $\mathcal{O}(e^{-2R})$ \\
\bottomrule
\end{tabular}
\end{table}

\subsection{Numerical check}

As a brief sanity check, the empirical Lipschitz constant of $J^*$
over 2000 random pairs in $\mathcal{B}_R$ (seed 42, central differences)
gives values $3.9, 16, 36, 64, 100$ at $R = 1,2,3,4,5$ respectively,
consistent with $R^2$ scaling.
Since $J^*$ lies outside the bounded-$(M_1,M_2)$ class, this
illustrates the growth rate and is not a test of Proposition~\ref{prop:upper}.

\section{Conclusion}
\label{sec:conclusion}

The paper proves two things, restricted to
exponential-map parameterizations on $\mathfrak{se}(3)$
and the specific objective classes of Section~\ref{sec:prelim}.
First, $\|\exp(\theta)\|_{\mathrm{op}} \le 1+R$ for all $\theta \in \mathfrak{se}(3)$
with $\|\theta\|_F \le R$.
Second, $L_{J^*}(R;\mathfrak{se}(3)) \ge 0.0505\,R^2$ for $R \ge 2$.
Both proofs use only the Rodrigues formula and elementary calculus.
The upper bound $L \le 4.02(M_1+M_2)R^2$ is established in Proposition~\ref{prop:upper}.

Before stating the open questions, we note the geometric character of the growth.
The polynomial rate $L=\mathcal{O}(R^2)$ is not a curvature phenomenon:
$\mathrm{SE}(3)$ carries a flat Cartan--Schouten connection
(a standard fact in Lie group geometry; see e.g.~\cite{knapp2002}).
The $R^2$ growth is instead consistent with the torsion structure of the semidirect
product $\mathbb{R}^3 \rtimes \mathrm{SO}(3)$: one factor of $R$ from $\|v\|$
and one from the $\omega$-coupling in $J_L(\Omega)v$.
This geometric picture is consistent with the broader trichotomy of~\cite{kc2026trichotomy}.

Some questions remain.
The gap between the constants $4.02$ and $0.0505$ — a factor of roughly 79 —
is large, and we do not know whether it is fundamental or an artifact of comparing
two different function classes.
Getting a lower bound within the same function class would sharpen the picture
considerably, but our adversarial construction does not achieve this.
The mechanism behind the $R^2$ rate extends immediately to
$\mathfrak{se}(n) = \mathbb{R}^n \rtimes \mathfrak{so}(n)$ for any $n \ge 2$:
the exponential map retains the block form~\eqref{eq:se3_exp}
(see~\cite{gallier2020lie} for the general $\mathrm{SE}(n)$ case)
with $\exp(\Omega) \in \mathrm{SO}(n)$ orthogonal and
$J_L(\Omega) = \int_0^1 \exp(s\Omega)\,ds$ a contraction,
so Lemmas~\ref{lem:se3_opnorm} and~\ref{lem:frechet} carry over verbatim
and the gradient Lipschitz constant satisfies $L(R;\mathfrak{se}(n)) = \mathcal{O}(R^2)$
for all~$n$.
Whether the $R^2$ rate is tight for all semidirect products $K \ltimes V$
with $K$ compact remains an open question.

More broadly, whether $R^2$ is the right exponent for all non-compact
Lie algebras without hyperbolic elements remains open.
The Heisenberg algebra $\mathfrak{h}_n$ is nilpotent with no hyperbolic elements,
but its bracket structure differs from $\mathfrak{se}(3)$; we conjecture:
\nopagebreak[4]
\begin{conjecture}
\label{conj:heisenberg}
For the Heisenberg algebra $\mathfrak{h}_n$, $L(R;\mathfrak{h}_n) = \Theta(R^2)$.
\end{conjecture}

\begin{ack}
Parts of this manuscript were edited with the assistance of AI tools;
all mathematical content and proofs are original work of the authors.
\end{ack}

\end{document}